\documentclass{article}

\makeatletter
\newcommand{\address}[1]{\gdef\@address{#1}}
\newcommand{\email}[1]{\gdef\@email{\url{#1}}}
\newcommand{\@endstuff}{\par\vspace{\baselineskip}\noindent\small
\begin{tabular}{@{}l}\scshape\@address\\\textit{E-mail address:} \@email\end{tabular}}
\AtEndDocument{\@endstuff}
\makeatother

\title{Haar measure for non-Hausdorff \\ locally compact groups}
\author{Lisa Valentini}
\date{\today}
\address{DIMA, Università di Genova, Via Dodecaneso 35, 16146 Genova, Italy}
\email{s4600785@studenti.unige.it}

\date{2 November 2024}

\usepackage{mathtools}
\usepackage{enumitem}
\usepackage{bbm}
\usepackage{amsthm}
\usepackage{bbding}
\usepackage[english]{babel}
\usepackage[square, numbers]{natbib}
\usepackage[breaklinks,hidelinks]{hyperref}
\usepackage{amssymb}

\DeclareMathOperator\spt{spt}

\newcommand{\numberset}{\mathbbm}
\newcommand{\N}{\numberset{N}}

\newcommand{\Q}{\numberset{Q}}
\newcommand{\R}{\numberset{R}}

\newtheorem{theorem}{Theorem}[section]
\newtheorem{lemma}[theorem]{Lemma}

\newtheorem{proposition}[theorem]{Proposition}

\newtheorem{manualtheoreminner}{Theorem}
\newenvironment{manualtheorem}[1]{%
  \renewcommand\themanualtheoreminner{#1}%
  \manualtheoreminner
}{\endmanualtheoreminner}

\theoremstyle{definition}
\newtheorem{definition}[theorem]{Definition}
\newtheorem*{remark}{Remark}
\newtheorem*{notation}{Notation}
\newtheorem{example}[theorem]{Example}

\begin{document}

\maketitle

\begin{abstract}
    The paper describes two possible ways of extending the definition of Haar measure to non-Hausdorff locally compact groups. The first one forces compact sets to be measurable: with this construction, a counterexample to the existence of the Haar measure is provided. The second one makes use of closed compact sets instead of compact sets in the definition of Radon measure: this way, the classical theorems of existence and uniqueness of the Haar measure can be generalised to locally compact groups, not necessarily Hausdorff.
\end{abstract}

\section*{Introduction}

A Haar measure on a Hausdorff topological group is a nonzero translation-invariant Radon measure; on a locally compact Hausdorff group, the Haar measure exists and it is unique up to scalar multiplication. 
Moreover, it is of general knowledge that these two results can be somehow generalised to the non-Hausdorff case. Nevertheless, this fact is usually stated without proof and it is considered a folklore result.

Many authors in the literature, e.g. Folland \cite[Chapter 11]{bi:folland} and Hewitt and Ross \cite[Theorem 5.36]{bi:hewitt}, focus on Hausdorff topological groups, briefly justifying that this is not a loss of generality. In fact, some sources quote the following construction as a way to generalise Haar measure to those topological groups that are not Hausdorff:
given a locally compact group $G$, the quotient $G / \overline{\{e_G\}}$ is locally compact and Hausdorff, thus it admits a Haar measure $\mu$; given the quotient map $\pi \colon G \to G / \overline{\{e_G\}}$, the assignment $\overline{\mu}(E)\coloneqq \mu(\pi(E))$ for every $E\in\mathcal{B}(G)$ defines a Haar measure on $G$.
The problem in this statement concerns the definition itself of a Haar measure, which classically requires the underlying space to be Hausdorff in order to guarantee that compact sets are Borel sets.

Our first issue, then, is re-defining the Haar measure for not necessarily Hausdorff topological groups. In this paper, two possible ways are considered: they both arise as extremely natural generalisations of the common definition of Haar measure on Hausdorff topological groups, but just one of them leads to the results of existence and uniqueness (with the only assumption of the group being locally compact). They develop as follows.

The first option that is taken into account forces compact sets to be measurable by defining a Radon measure on the wider $\sigma$-algebra generated by compact sets and open sets: this attempt fails in ensuring the existence of a Haar measure, even in case the topological group is compact, and an example of this situation is provided (Example \ref{ex:gen.1}).

The second way consists in replacing the collection of compact sets with the collection of closed compact sets, which are automatically Borel sets, in the definition of Radon measure, and more specifically of local finiteness and inner regularity (Definition \ref{def:Radon}): by implementing the consequent definition of Haar measure, on a non necessarily Hausdorff locally compact group the Haar measure exists and is unique up to scalar multiplication. The present paper is mainly devoted to give a direct proof of these facts, without involving quotient groups.

The main theorems read as follows.

\begin{manualtheorem}{\ref{th:existence}}
    Let $G$ be a locally compact topological group. Then there exists a left Haar measure on $G$.
\end{manualtheorem}

\begin{manualtheorem}{\ref{th:uniqueness}}
    Let $G$ be a locally compact topological group, and let $\mu$ and $\mu'$ be left Haar measures on $G$. Then there exists a constant $a>0$ such that $\mu'=a\mu$.
\end{manualtheorem}

The proofs provided here are an adaptation of the well-known ones valid for locally compact Hausdorff groups, which are described, e.g., by Cohn \cite{bi:cohn}. The generalisation is based on four main facts: a topological group is always regular (see, e.g., \cite[Proposition 3.1.15]{bi:dikr}); in a regular space, compact sets and closed sets can be separated (Lemma \ref{th:compact_closed_disjoint}); in a locally compact regular space, every point admits as closed compact neighborhood, and every compact set has a closed compact neighborhood (Lemma \ref{th:c_o_cc}).

Moreover, with the definition of Haar measure adopted here, the construction via quotients of a Haar measure on a non-Hausdorff locally compact group actually works, as shown in Proposition \ref{th:quotient}.
The main difference is that this construction provides a way of proving existence and uniqueness of the Haar measure on locally compact groups by assuming that these results are valid on locally compact Hausdorff groups; on the contrary, the proof given in Theorems \ref{th:existence} and \ref{th:uniqueness} does not rely on this fact.

A standard example of a non-Hausdorff locally compact group is the product of groups $X \times G$, where $X$ is a Hausdorff locally compact group and $G$ is a non-trivial indiscrete group, endowed with the product topology: using the generalised definition of Haar measure, we show that the Haar measure on $X \times G$ is the composition of the Haar measure on $X$ with the projection on $X$ (Example \ref{ex:product}).

The paper is structured as follows.
In Section \ref{sec:std_def}, standard definitions and theorems concerning the Haar measure on locally compact Hausdorff groups are quickly recalled, in order to underline the aspects that need to be adapted to the non-Hausdorff case.
In Section \ref{sec:new_def}, we discuss the two above-mentioned attempts of generalising the Haar measure to non-Hausdorff topological groups: Example \ref{ex:gen.1} shows that the first attempt does not lead to the desired result; conversely, Definitions \ref{def:regular} and \ref{def:Radon} introduce the Haar measure using closed compact sets to build local finiteness and inner regularity, and they are adopted for the rest of the paper.
Section \ref{sec:regular} hosts preliminary facts given by regularity and local compactness of topological spaces.
In Section \ref{sec:Haar}, Theorems \ref{th:existence}, \ref{th:uniqueness} and \ref{th:dx_sx} prove that on a locally compact group the Haar measure exists and is unique up to multiplication with a positive constant: applications are given in Example \ref{ex:product}.
Here, both results of existence and uniqueness are checked directly, that means without assuming that they are valid for locally compact Hausdorff groups. 
In Section \ref{sec:quot}, we check that the construction via quotients of a Haar measure on a (non-Hausdorff) locally compact group works; therefore, Propositions \ref{th:quotient} and \ref{th:quotient_uniq} offer a second way of proving existence and uniqueness of the Haar measure also in the non-Hausdorff case.

\vspace{0.5cm}

\noindent \rule{\textwidth}{0.4pt}

\section{Standard definitions} \label{sec:std_def}

As definitions are under examination themselves, we start by fixing terminology.

\begin{notation}
    Let $X$ be a topological space. The Borel $\sigma$-algebra on $X$ will be denoted $\mathcal{B}(X)$. The sets of real-valued continuous functions and of real-valued compactly supported continuous functions on $X$ will be denoted $C(X)$ and $C_c(X)$ respectively. The support of a real-valued function on $X$ will be denoted $\spt(f)$. Given a subset $Y\subseteq X$, the indicator function of $Y$ will be denoted $\mathbbm{1}_Y$.

    Given $G$ a topological group, we will denote with $\tau_G$ and $e_G$ its topology and its identity respectively, unless specified otherwise.
\end{notation}

We make use of the standard definitions of regular topological space and normal topological space, as given, e.g., by Kelley \cite[p.~112]{bi:kelley}: they do not require the topological space to be $T_1$.

\begin{definition}          \label{def:loc_comp}
    A topological space $X$ is said to be \emph{locally compact} if every point of $X$ admits a compact neighborhood.
\end{definition}

Clearly, every point of a regular locally compact space admits a closed compact neighborhood (i.e. a regular locally compact space is \emph{strongly locally compact}, using the terminology by Steen and Seebach \cite{bi:steen_seebach}).

The following is standard vocabulary for Hausdorff topological groups.

\begin{definition}          \label{def:regular_H}
	Let $X$ be a Hausdorff topological space and $\mathcal{M}$ be a $\sigma$-algebra on $X$ such that $\mathcal{M}\supseteq \mathcal{B}(X)$. A positive measure $\mu\colon\mathcal{M}\to[0,+\infty]$ is said to be:
	\begin{enumerate}[label=(\roman*)]
	    \item \emph{outer regular on $E\in\mathcal{M}$} if \; $\mu(E)=\inf \{\mu(U)\mid U\supseteq E \text{ open}\}$; \label{def:out}
		\item \emph{inner regular on $E\in\mathcal{M}$} if \; $\mu(E)=\sup \{\mu(K)\mid K\subseteq E \text{ compact}\}$.
	\end{enumerate}
\end{definition}

\begin{definition}          \label{def:Radon_H}
    Let $X$ be a Hausdorff topological space. A \emph{Radon measure} on $X$ is a positive Borel measure that is finite on compact sets, outer regular on all Borel sets and inner regular on all open sets.
\end{definition}

\begin{definition}          \label{def:translation_invariant}
    Let $G$ be a topological group. A positive Borel measure $\mu$ on $G$ is said \emph{left-translation-invariant} (resp. \emph{right-translation-invariant}) if for every $A\in\mathcal{B}(G)$ and for every $g\in G$ one has $\mu(gA)=\mu(A)$ (resp. $\mu(Ag)=\mu(A)$).
\end{definition}

\begin{definition}          \label{def:Haar_H}
    Let $G$ be a Hausdorff topological group. A \emph{left} (resp. \emph{right}) \emph{Haar measure} on $G$ is a nonzero left-translation-invariant (resp. right-translation-invariant) Radon measure.
\end{definition}

In this context, the Hausdorff hypothesis ensures compact sets to be Borel sets, and then Definition \ref{def:regular_H}, \ref{def:Radon_H} and \ref{def:Haar_H} to be consistent. Together with the local compactness of the topological group, it allows to build a Radon measure that is eventually checked to be a Haar measure, and to prove its uniqueness up to a multiplicative constant.
The classical theorems concerning the Haar measure on locally compact Hausdorff groups state as follows.

\begin{theorem}             \label{th:existence_H}
	Let $G$ be a locally compact Hausdorff topological group. Then there exists a left Haar measure on $G$.
\end{theorem}

\begin{theorem}             \label{th:uniqueness_H}
	Let $G$ be a locally compact Hausdorff topological group, and let $\mu$ and $\mu'$ be left Haar measures on $G$. Then there exists $a>0$ such that $\mu'=a\mu$.
\end{theorem}

It is immediate to see that Theorem \ref{th:existence_H} is no longer true for a Hausdorff topological group that is not locally compact. For instance, any countably infinite Hausdorff topological group having infinite compact sets (e.g. the group of rational numbers $\Q$ with the topology inherited by the Euclidean topology on $\R$) does not have a Haar measure; see Example \ref{ex:gen.1}(1).

\begin{remark}
    Since we are interested in applying Definition \ref{def:Radon_H} only to locally compact groups, we will use the terminology \emph{local finiteness} to refer to finiteness on compact sets. 
\end{remark}

\section{Generalised definitions}   \label{sec:new_def}

Our aim is generalising the Haar measure to topological groups that are not necessarily Hausdorff.
First, in order to preserve the notions of local finiteness and inner regularity of a measure $\mu$ on a (not necessarily Hausdorff) topological space $X$, the domain of $\mu$ must contain a collection of sets which is closed under finite unions.
As seen in Definition \ref{def:regular_H}, in the case of a Hausdorff topological space this role is played by the collection of compact sets.

The first idea that can be examined is rephrasing that definition by forcing compact sets to be measurable: that is, by considering $\mathcal{BK}(X)$ the $\sigma$-algebra generated by open subsets and compact subsets of $X$, and by defining inner regularity and outer regularity as given in Definition \ref{def:regular_H} for a measure $\mu \colon \mathcal{M} \to [0, + \infty]$ where $\mathcal{M} \supseteq \mathcal{BK}(X)$.
Then, new definitions of Radon measure and Haar measure, whose domain is now $\mathcal{BK}(X)$, follow straightforwardly.

Nevertheless, the following example shows that in these terms the existence of a Haar measure on a topological group $G$ is not guaranteed, not even if $G$ is compact.

\begin{example}                 \label{ex:gen.1}
    Let $(G,\tau_G)$ be a non-trivial topological group.
    \begin{enumerate}[label=(\arabic*)]
        \item If $G$ is countably infinite and $(G,\tau_G)$ has infinite compact sets, then it admits no (left or right) Haar measure in the generalised terms above.
        \item If $\tau_G$ is the indiscrete topology, then $(G,\tau_G)$ is compact and admits no (left or right) Haar measure in the generalise terms above.
    \end{enumerate}
\end{example}
\begin{proof}
    Suppose $\mu \colon \mathcal{BK}(G) \to [0,+\infty]$ is a Haar measure. Let $a\ge 0$ be the value of the finite measure of singletons, which does not depend on the singleton due to translation invariance. 
    
    (1)  If $a>0$, then $\mu(K)=+\infty$ for every infinite compact set $K\subseteq G$. Then $a=0$, which yields $\mu(G)=0$ because $G$ is countable.

    (2) For the outer regularity of the measure, one gets $a=\mu(G)$ because the topology is indiscrete. Moreover, $\mu(G)\ge 2a$ since $G$ has at least two elements. Then $a=0$ and $\mu(G)=0$.
\end{proof}

\medskip

\begin{remark}
    The proof of Example \ref{ex:gen.1}(1) shows that, more in general, a countably infinite topological group having infinite compact sets admits no translation-invariant measure that is finite on compact sets. As suggested by the referee, the same statement holds for an indiscrete infinite abelian group $G$ too (which is Example \ref{ex:gen.1}(2) when $G$ is also infinite and abelian). This can be proved using the fact that infinite compact abelian groups have a subgroup of infinite countable index \cite[\S16.13(b),(c)]{bi:hewitt}. Indeed, if $H$ is such a subgroup of $G$, then $G= \sqcup_{n\in\N} x_n+H$ is the disjoint union of the cosets $x_n+H$ for some sequence $(x_n)_{n\in\N}$ of points of $G$. All those cosets are compact since the topology is indiscrete, and therefore they have finite measure $b\ge0$ which does not depend on $n\in\N$ due to the translation invariance. But $\mu(G)<+\infty$ because $G$ is compact, which yields $b=0$ and $\mu(G)=0$. In fact, the same argument works for all indiscrete infinite groups admitting a subgroup of infinite countable index: more details can be found in \cite{bi:extra_index}.
\end{remark}

Given Example \ref{ex:gen.1}, we are induced to abandon the previous attempt of generalising Haar measure.

The second possible construction of Haar measure for non-Hausdorff groups considered in this paper makes use of the collection of closed compact sets, which is then automatically included in the Borel $\sigma$-algebra. With this generalisation, we get an extension of the theorems concerning existence and uniqueness of the Haar measure also to non-Hausdorff locally compact groups. The following two sections are devoted to details concerning this generalisation.

\begin{definition}          \label{def:regular}
    Let $X$ be a topological space and $\mathcal{M}$ be a $\sigma$-algebra on $X$ such that $\mathcal{M}\supseteq \mathcal{B}(X)$. A positive measure $\mu\colon\mathcal{M}\to[0,+\infty]$ is said to be:
	\begin{enumerate}[label=(\roman*)]
            \item \emph{outer regular on $E\in\mathcal{M}$} as in Definition \ref{def:regular_H}\ref{def:out};
		\item \emph{inner regular on $E\in\mathcal{M}$} if \, $\mu(E)=\sup \{\mu(K)\mid K\subseteq E \text{ closed compact}\}$.
	\end{enumerate}
\end{definition}

\begin{definition}          \label{def:Radon}
    Let $X$ be a topological space. A \emph{Radon measure} on $X$ is a positive Borel measure that is finite on closed compact sets, outer regular on all Borel sets and inner regular on all open sets.
\end{definition}

Definition \ref{def:Radon} is coherent with Definition \ref{def:Radon_H}. Then, a generalized definition of Haar measure -- which is formally identical to Definition \ref{def:Haar_H} but is also valid in the non-Hausdorff case -- follows straightforwardly.

\begin{remark}
    Being translation-invariant, a left (resp. right) Haar measure $\mu$ on a topological group $G$ keeps the property of inducing left- (resp. right-) translation-invariant integrals, that is $\int_G f(gx) d\mu(x) = \int_G f(x) d\mu(x)$ (resp. $\int_G f(xg) d\mu(x) = \int_G f(x) d\mu(x)$) for every $g\in G$ and $f\in \mathcal{L}^1(\mu)$.
\end{remark}

From this point onward, we will refer to this extended definitions of Radon measure and Haar measure.

\section{Locally compact regular spaces}    \label{sec:regular}

This section collects the preliminary results that will be used in Section \ref{sec:Haar}. Most of them are adaptions of standard theorems.

Since we are interested in topological groups, which are regular (see, e.g., \cite[Proposition 3.1.15]{bi:dikr}), we will focus on regular spaces.
We begin with the already known possibility of separating a compact set and a closed set in a regular space: see, e.g., \cite[\S5 Theorem 10]{bi:kelley}.

\begin{lemma}           \label{th:compact_closed_disjoint}
    Let $X$ be a regular topological space, and let $A,B\subseteq X$ be disjoint subsets, with $A$ compact and $B$ closed. Then there exist $U$ and $V$ disjoint open sets such that $A\subseteq U$ and $B\subseteq V$. In particular, if $X$ is regular and compact, than it is normal.
\end{lemma}

The following technical lemma is a slight variant of a result that can be found, e.g., in \cite[Lemma 7.1.10]{bi:cohn}.

\begin{lemma}   	   \label{th:compact_in_opens}
	Let $X$ be a regular topological space. Consider a closed compact set $K$ and two open sets $U_1$ and $U_2$ of $X$ such that $K \subseteq U_1 \cup U_2$. Then there exist $K_1$ and $K_2$ closed compact sets of $X$ such that $K_1 \subseteq U_1$, $K_2\subseteq U_2$ and $K=K_1 \cup K_2$.
\end{lemma}
\begin{proof}
	Let $L_i\coloneqq K\setminus U_i$ for $i=1,2$: they are closed, compact and disjoint.
    By Lemma \ref{th:compact_closed_disjoint} there exist $V_1$, $V_2$ disjoint open sets such that $L_i \subseteq V_i$ for $i=1,2$. Let $K_i\coloneqq K\setminus V_i$ for $i=1,2$, which are closed and compact. It is immediate to check that $K_i\subseteq U_i$ for $i=1,2$, and $K_1\cup K_2= K$.
\end{proof}

\medskip

We now deal with some properties linked to local compactness.
We start with a fact useful to replace compact sets with closed compact sets when dropping the Hausdorff hypothesis.

\begin{lemma}           \label{th:c_o_cc}
    Let $X$ be a regular locally compact topological space and let $K$ be a compact subset of $X$. Then there exist an open set $U$ and a closed compact set $L$ such that $K\subseteq U \subseteq L$.
\end{lemma}
\begin{proof}
    By regularity and local compactness, for every $x\in K$ there exist a closed compact set $L_x$ and an open set $U_x$ such that $x\in U_x\subseteq L_x$. Then there exist $x_1,\dots,x_n\in K$ such that $K\subseteq \bigcup_{i=1}^n U_{x_i} \subseteq \bigcup_{i=1}^n L_{x_i}$, where $U \coloneqq \bigcup_{i=1}^n U_{x_i}$ is open and $L\coloneqq \bigcup_{i=1}^n L_{x_i}$ is closed and compact.
\end{proof}

\medskip

We also need to replace the well known statement of the Urysohn lemma for locally compact Hausdorff spaces with an analogous version concerning regular locally compact spaces: a proof given by Kelley in  \cite[\S5 Theorem 18]{bi:kelley} provides Lemma \ref{th:K<U} and, by just a few more specifications, the desired theorem too.

\begin{lemma}               \label{th:K<U}
    Let $X$ be a regular locally compact topological space, and consider an open set $U$ and a closed compact set $K$ in $X$ such that $K\subseteq U$. Then there exists an open set $V$ with compact closure such that $K \subseteq V \subseteq \overline{V} \subseteq U$.
\end{lemma}

\begin{theorem}             \label{th:urysohn}
    Let $X$ be a regular locally compact topological space, and consider an open set $U$ and a closed compact set $K$ in $X$ such that $K\subseteq U$. Then there exists $g\in C_c(X)$ such that $\mathbbm{1}_K\leq g \leq \mathbbm{1}_U$ e $\spt(g)\subseteq U$.
\end{theorem}
\begin{proof}
    Let $V$ be an open neighborhood of $K$ as in Lemma \ref{th:K<U}. See \cite[\S5 Theorem 18]{bi:kelley}: it provides a continuous function $f \colon X \to [0,1]$ which is zero on $K$ and one on $X\setminus \overline{V}$. The function we are looking for is $g\coloneqq 1-f$, which is also compactly supported in $U$, since $\spt(g) \subseteq \overline{V} \subseteq U$, where $\overline{V}$ is compact and closed. 
\end{proof}

\medskip

The previous result allows to adapt to regular locally compact spaces the well known Riesz representation theorem for positive linear functionals, specifically the part concerning the uniqueness of the representing measure.

\begin{theorem}             \label{th:riesz}
    Let $X$ be a regular locally compact topological space, and let $\mathcal{M}$ be a $\sigma$-algebra on $X$ such that $\mathcal{M}\supseteq\mathcal{B}(X)$. Let $\Lambda\colon C_c(X)\to\R$ be a positive linear functional and $\mu_1,\mu_2\colon\mathcal{M}\to[0,+\infty]$ be Radon measures such that for every $f\in C_c(X)$ the equality $\int_X f \, d\mu_i=\Lambda(f)$ holds for $i=1,2$. Then $\mu_1=\mu_2$.
\end{theorem}
\begin{proof}
    By inner and outer regularity, it is enough to prove that $\mu_1$ and $\mu_2$ coincide on closed compact sets. Given a closed compact set $K$ and an open set $U$ such that $K\subseteq U$, by Theorem \ref{th:urysohn} there exists $f\in C_c(X)$ such that $\mathbbm{1}_K\leq f \leq \mathbbm{1}_U$ and $\spt(f)\subseteq U$.
    Then $\mu_1(K) \leq \Lambda(f)  \leq \mu_2(U)$.    
    Taking the lower bound in the right side of the inequality upon the open sets containing $K$, one has $\mu_1(K)\leq\mu_2(K)$.
    Equality follows by switching $\mu_1$ and $\mu_2$. 
\end{proof}

\medskip

Lastly, we will need to compute integrals on a product space with respect to Radon measures.
The next lemma can be found, e.g., in \cite[Lemma 7.6.3]{bi:cohn}, while the following proposition, that makes use of the lemma, is proved as a slight variant of \cite[Proposition 7.6.4]{bi:cohn}:  in order to avoid redundancy, some technical passages make explicit reference to the original proof.  

\begin{lemma}               \label{th:lemma_int}
	Let $X$ and $Y$ be topological spaces, with $Y$ being compact, and let $f:X\times Y \to \R$ be continuous. Then for every $x_0\in X$ and for every $\epsilon>0$ there exists an open neighborhood $U_0$ of $x_0$ such that for every $x\in U_0$ and for every $y\in Y$ one has $|f(x,y)-f(x_0,y)|<\epsilon$.
	\label{cont_compatto}
\end{lemma}

\begin{proposition}             \label{th:product_measure}
    Let $X$ and $Y$ be regular locally compact topological spaces, and let $\mu$ and $\lambda$ be Radon measures on $X$ and $Y$ respectively. Then for every function $f\in C_c(X\times Y)$ one has the following:
	\begin{enumerate}[label=(\alph*)]
		\item $\big[y \mapsto f(x,y)\big] \in C_c(Y)$ for every $x\in X$, and $\big[x \mapsto f(x,y)\big] \in C_c(X)$ for every $y\in Y$;
		\item $\Bigl[ x \mapsto \displaystyle\int_Y f(x,y) \, d\lambda(y) \Bigr] \in C_c(X)$, and $\Bigl[ y \mapsto \displaystyle\int_X f(x,y) \, d\mu(x) \Bigr] \in C_c(Y)$;
		\item $\displaystyle\int_X \int_Y f(x,y) \, d\lambda(y) d\mu(x) = \int_Y \int_X f(x,y) \, d\mu(x) d\lambda(y)$.
	\end{enumerate}
\end{proposition}

\begin{proof}
    Let $f\in C_c(X\times Y)$: Consider the sets $K \coloneqq \spt(f)$, $K_1 \coloneqq \pi_1(K)$ and $K_2 \coloneqq \pi_2(K)$, which are all compact. Let $K_1'$ and $K_2'$ be compact closed sets containing $K_1$ and $K_2$ respectively, as given by Lemma \ref{th:c_o_cc}.
    
    {\it (a)}
        Given $x\in X$, the function $y \mapsto f(x,y)$ is continuous and vanishes outside $K_2$, and then outside the closed set $K_2'$: therefore, its support lays in  $K_2'$ and thus it is compact. Given $y\in Y$, the proof for $x \mapsto f(x,y)$ is similar.

    {\it (b)}
        Let $x_0\in X$ and $\epsilon>0$.
		Applying Lemma \ref{th:lemma_int} to $f|_{X \times K_2'}$, one finds an open neighborhood $U_0$ of $x_0$ such that $| f(x,y) - f(x_0,y) | < \epsilon$ for every $x\in U_0$ and $y\in K_2'$. Therefore, for every $x\in U_0$ one has:
        \[\left|\int_Y f(x,y) \, d\lambda(y) - \int_Y f(x_0,y) \, d\lambda(y)\right| \leq \epsilon\lambda(K_2').\]
        Then, the function $x \mapsto \int_Y f(x,y) \, d\lambda(y)$ is continuous. Moreover, it vanishes outside $K_1$, and then outside the closed set $K_1'$: thus its support is compact since contained in the compact set $K_1'$. The proof for $y\mapsto \int_X f(x,y) \, d\mu(x)$ is similar.

    {\it (c)} 
        Let $\epsilon>0$.
        Applying the procedure shown in \cite[Proposition 7.6.4]{bi:cohn} to the closed compact set $K_1'$,	one finds $x_1,\dots,x_n\in K_1'$ and $A_1,\dots,A_n$ disjoint Borel sets such that $K_1' = \bigcup_{k=1}^n A_k$ and $| f(x,y) - f(x_k,y) | < \epsilon$ for every $x\in A_k$ and $y\in K_2'$, for every $k\in\{1,\dots,n\}$.
        Define:
        \[g \colon X \times Y \to \mathbbm{R}, \quad (x,y) \mapsto \sum_{k=1}^n \mathbbm{1}_{A_k}(x) f(x_k,y)
		\,\text{.} \]
        Both iterate integrals of $g$ are equal to $ \sum_{k=1}^n \mu(A_k) \int_Y f(x_k,y)d\lambda(y)$, because $[y \mapsto f(x_k,y)]\in C_c(Y)$ and $\mu(A_k)\leq\mu(K_1')<+\infty$ hold for every $k\in\{1,\dots,n\}$, since $K_1'$ is a closed compact set.
        The functions $f$ and $g$ vanish outside $K_1' \times K_2'$; on the contrary, for every $(x,y)\in K_1' \times K_2'$ there exists a unique index $i\in\{1,\dots,n\}$ such that $x\in A_i$, which implies that $|f(x,y)-g(x,y)|=|f(x,y)-f(x_i,y)|<\epsilon$. Then:
		\[\bigg|\int_Y \int_X f(x,y) \, d\mu(x) d\lambda(y) - \int_Y \int_X g(x,y) \, d\mu(x) d\lambda(y)\bigg| \leq \epsilon\mu(K_1')\lambda(K_2')
		\]
		\[\bigg|\int_X \int_Y f(x,y) \, d\lambda(y) d\mu(x) - \int_X \int_Y g(x,y) \, d\lambda(y) d\mu(x)\bigg| \leq \epsilon\mu(K_1')\lambda(K_2')
		\, \text{.} \]
		Since the iterate integrals of $g$ are equal, one gets:
		\[\bigg|\int_Y \int_X f(x,y) \, d\mu(x) d\lambda(y) - \int_X \int_Y f(x,y) \, d\lambda(y) d\mu(x)\bigg| \leq 2\epsilon\mu(K_1')\lambda(K_2')
		\, \text{.}\]
        Since $\epsilon$ is arbitrary, the claim follows.  \qedhere
\end{proof}

\medskip

\section{Haar measure on locally compact groups}    \label{sec:Haar}

In this section we prove the theorems of existence and uniqueness of the Haar measure on a locally compact group. The structure of the proofs is faithful to the classical theorems concerning locally compact Hausdorff groups; Cohn \cite[\S 9.2]{bi:cohn} is taken as reference. The results shown in Section \ref{sec:regular} replace the standard ones valid in the Hausdorff case.

We start by showing the existence of a left Haar measure on a locally compact topological group.

\begin{theorem}         \label{th:existence}
	Let $G$ be a locally compact topological group. Then there exists a left Haar measure on $G$.
\end{theorem}
\begin{proof}
    The proof is structured as in \cite[Theorem 9.2.2]{bi:cohn}, and considers the collection of closed compact sets instead of the collection of compact sets.
\end{proof}

\medskip

In order to prove the uniqueness of the left Haar measure, two more results need to be generalised. Recall that in a topological group the product of a compact set and a closed set is closed (see, e.g., \cite[Theorem 4.4]{bi:hewitt}).

\begin{proposition}         \label{th:cont_trans}
    Let $G$ be a locally compact group and $\mu$ be a Radon measure on $G$. Then for every $f\in C_c(G)$ one has:
	\begin{enumerate}[ label=(\alph*)]
		\item $[x \mapsto f(xg)]\in C_c(G)$ and $[x \mapsto f(gx)]\in C_c(G)$ for every $g\in G$;
		\item $ \Bigl[ g \mapsto \displaystyle\int_G f(gx) d\mu(x) \Bigr] \in C(G)$ and $ \Bigl[ g \mapsto \displaystyle\int_G f(xg) d\mu(x) \Bigr] \in C(G)$.
	\end{enumerate}
\end{proposition}
\begin{proof}
    Claim (a) is straightforward. Claim (b) is proved as in \cite[Corollary 9.1.6]{bi:cohn}: we recall the main steps to prove that $g\mapsto \int_G f(xg)d\mu(x)$ is continuous in a point $g_0\in G$.
    Let $K\coloneqq \spt(f)$ and let $U$ be an open neighborhood of $g_0$ with compact closure. The set $K\overline{U}^{-1}$ is compact because $K\times\overline{U}^{-1}$ is so, and it is closed because $K$ is compact and $\overline{U}^{-1}$ is closed. Then $\mu(K\overline{U}^{-1})<+\infty$, since $\mu$ is finite on closed compact sets. The proof proceeds as in \cite[Corollary 9.1.6]{bi:cohn}: fixed $\epsilon>0$, by applying Lemma \ref{th:lemma_int} one finds an open neighborhood $V$ of $g_0$ such that for every $g\in V$ the function $x \mapsto f(xg) - f(xg_0)$ has absolute value controlled by $\epsilon$ and vanishes outside $K\overline{U}^{-1}$; therefore, for every $g\in V$ one has:
    \[\left|\int_G f(xg) \, d\mu(x)-\int_G f(xg_0) \, d\mu(x)\right| \leq \epsilon\mu(K\overline{U}^{-1}).\]
    Thus the function $g\mapsto \int_G f(xg)d\mu(x)$ is continuous.
\end{proof}

\medskip

\begin{proposition}             \label{th:Haar_positive}
	Let $G$ be a locally compact topological group, and let $\mu$ be a (left or right) Haar measure on $G$. Then:
	\begin{enumerate}[label=(\alph*)] 
		\item there exists a closed compact set $K$ such that $\mu(K)>0$;
		\item for every non empty open set $U$ one has $\mu(U)>0$;
		\item for every $f\in C_c(G)$ non negative and nonzero one has $\int_G f d\mu>0$.
	\end{enumerate}
\end{proposition}
\begin{proof}
	The proof follows exactly as in \cite[Theorem 9.2.5]{bi:cohn}: claim (a) holds since the Radon measure $\mu$ is nonzero, and claims (b) and (c) are straightforward consequences.
\end{proof}

\medskip

We are now ready to prove the uniqueness of the Haar measure on locally compact groups.

\begin{theorem}         \label{th:uniqueness}
	Let $G$ be a locally compact topological group, and let $\mu$ and $\mu'$ be left Haar measures on $G$. Then there exists a constant $a>0$ such that $\mu'=a\mu$.
\end{theorem}
\begin{proof}
    The proof follows as in \cite[Theorem 9.2.6]{bi:cohn}.

    Let $K$ be a non-empty closed compact set such that $\mu(K)>0$, that exists by Proposition \ref{th:Haar_positive}(a). By applying Theorem \ref{th:urysohn} to $K$ and $G$, one gets a function $g\in C_c(G)$ which is nonzero and non negative. Fix $f\in C_c(G)$. Let $L\coloneqq \spt(f)$ and $F\coloneqq \spt(g)$.
    Consider the function:
    \[h\colon G\times G \to \R, \quad (x,y)\mapsto \frac{f(x) \, g(yx)}{\int_G g(tx) \, d\mu'(t)}
	\]
    which is well defined by Propositions \ref{th:cont_trans}(a) and \ref{th:Haar_positive}(c), and continuous by Proposition \ref{th:cont_trans}(b). It is also compactly supported, since $\spt(h)\subseteq L\times FL^{-1}$ which is closed and compact.

    The Haar measures $\mu$ and $\mu'$ induce left translation-invariant integrals. By applying Proposition \ref{th:product_measure}(c) to $h\in C_c(G\times G)$ (where the first and the second factors of $G \times G$ are endowed with $\mu$ and $\mu'$ respectively) and the changes of variables $x \mapsto y^{-1}x$ and $y \mapsto xy$, one proceeds as in \cite[Theorem 9.2.6]{bi:cohn} to get:

    \begin{align*}
	\int_G f(x) \, d\mu(x)
        & = \int_G \biggl[ f(x) \; \frac{\int_G g(yx) \, d\mu'(y)}{\int_G g(tx) \, d\mu'(t)} \biggr] d\mu(x) \\
        & =	\int_G \biggl[ \int_G h(x,y) \, d\mu'(y) \biggr] d\mu(x) \\
        & =	\int_G \biggl[ \int_G h(y^{-1}x,y) \, d\mu(x) \biggr] d\mu'(y) \\
        & =	\int_G \biggl[ \int_G h(y^{-1},xy) \, d\mu'(y) \biggr] d\mu(x) \\
        & =	\int_G \biggl[ \int_G \frac{f(y^{-1}) \, g(x)}{\int_G g(ty^{-1}) \, d\mu'(t)} \, d\mu'(y) \biggr] d\mu(x) \\
        & = \biggl( \int_G g(x) \, d\mu(x) \biggr) \biggl( \int_G \frac{f(y^{-1})}{\int_G g(ty^{-1})\, d\mu'(t)} \, d\mu'(y) \biggr) \,.
    \end{align*}

    Analogously, one gets:
	\[ \int_G f(x) \, d\mu'(x) = \biggl( \int_G g(x) d\mu'(x) \biggr) \biggl( \int_G \frac{f(y^{-1})}{\int_G g(ty^{-1}) \, d\mu'(t)} \, d\mu'(y) \biggr)
	\, \text{,}\]
	All the integrals involved are nonzero by Proposition \ref{th:Haar_positive}. Therefore:
    \[\frac{\int_G f(x) \, d\mu(x)}{\int_G g(x) \, d\mu(x)} = \frac{\int_G f(x) \, d\mu'(x)}{\int_G g(x) \, d\mu'(x)}
	\,\text{.}\]
	By defining $a \coloneqq \frac{\int_G g(x) \, d\mu'(x)}{\int_G g(x) \, d\mu(x)} > 0$, the equality $\int_G f \, d\mu' = a\int_G f \, d\mu$ holds for every $f\in C_c(G)$. By Theorem \ref{th:riesz}, one concludes that $\mu'= a \mu$.
\end{proof}

\medskip

Lastly, as in the classical proof, the following proposition ensures existence and uniqueness of the right Haar measure too, since right Haar measures and left Haar measures are in a one-to-one correspondence.

\begin{proposition}             \label{th:dx_sx}
	Let $G$ be a topological group, and let $\mu$ be a positive Borel measure on $G$. For every $E\in\mathcal{B}(G)$ define $\mu'(E)\coloneqq \mu(E^{-1})$. Then $\mu$ is a left (resp. right) Haar measure on $G$ if and only if $\mu'$ is a right (resp. left) Haar measure on $G$. 
\end{proposition}
\begin{proof}
    The proof follows as in \cite[Proposition 9.3.1]{bi:cohn}: 	if $\mu$ is a left Haar measure, $\mu'$ is a positive Borel measure which is nonzero and right-translation-invariant; it is also a Radon measure since the inversion map is an homeomorphism, and therefore $K\subseteq G$ is closed and compact if and only if $K^{-1}\subseteq G$ is closed and compact.
\end{proof}

\medskip
 
\begin{example}             \label{ex:product}
    Let $(X,\tau_X)$ be a Hausdorff locally compact group with (left or right) Haar measure $\mu_X$, and let $G$ be a non-trivial indiscrete group. The product of groups $X \times G$ with the product topology $\tau_{X\times G}$ is locally compact and non-Hausdorff. We have $\tau_{X\times G} = \{U\times G \mid U\in\tau_X\}$ and $\mathcal{B}(X\times G) = \{E\times G \mid E\in \mathcal{B}(X) \}$. The (left or right) Haar measure on $X \times G$ is given by:
    \[
        \mu_{X\times G} \colon \mathcal{B}(X \times G) \to [0,+\infty], \quad E \mapsto \mu(\pi_X(E))	
    \]
    where $\pi_X \colon X\times G \to X$ is the natural projection.
\end{example}

\begin{remark}
    In the case of a locally compact Hausdorff group, the generalised statements, that are Theorems \ref{th:existence} and \ref{th:uniqueness}, exactly coincide with their usual version.
\end{remark}

\section{The construction via quotients}    \label{sec:quot}

As said in the introduction, some sources in the literature suggest a way to generalise the Haar measure to non-Hausdorff groups through quotients. Using the definitions given in Section \ref{sec:new_def}, now it is easy to prove the consistency of that construction. See, e.g., \cite{bi:dikr} for the properties of topological groups that will be used.

Let $(G,\tau_G)$ be a topological group. Since the set $\overline{\{e_G\}}$ is a normal subgroup of $G$ \cite[Lemma 3.1.1(b)]{bi:dikr}, consider the group $G_H \coloneqq G / \overline{\{e_G\}}$ endowed with the quotient topology $\tau_{G_H}$, i.e. the final topology induced by the quotient map $\pi \colon G \to G_H$. We recall some basic facts.

\begin{enumerate}[label*=(\roman*)]
    \item \label{ls:G_H_T2}
    The topological group $G_H$ is Hausdorff \cite[Lemma 3.2.10(b)]{bi:dikr}. Moreover, if $G$ is locally compact, then $G_H$ is locally compact too \cite[Lemma 8.2.5(a)]{bi:dikr}.
    \item \label{ls:G_topology} The topology $\tau_G$ coincides with the initial topology on $G$ with respect to the quotient map $\pi$ \cite[Lemma 3.4.1(b)]{bi:dikr}, that is,
    \begin{equation*}
        \tau_G = \{ \pi^{-1}(V) \mid V\in \tau_{G_H} \}.
    \end{equation*}
    In particular, this implies
    \begin{equation*}
        \mathcal{B}(G) = \{ \pi^{-1}(F) \mid F\in \mathcal{B}(G_H) \} ,
    \end{equation*}
    and also that all closed compact subsets of $G$ are of the form $\pi^{-1}(C)$ for some compact set $C\subseteq G_H$.
    \item \label{ls:pi_open}
    The map $\pi$ is open \cite[Lemma 3.2.1(a)]{bi:dikr}, and therefore
    \begin{equation*}
        \tau_{G_H}=\{\pi(U) \mid U\in\tau_G\}.
    \end{equation*}
    In particular, this implies
    \begin{equation*}
        \mathcal{B}(G_H)=\{\pi(E) \mid E\in\mathcal{B}(G)\}.
    \end{equation*}
    Indeed, the collection $\{\pi(E) \mid E\in\mathcal{B}(G)\}$ is the $\sigma$-algebra generated by $\{\pi(U) \mid U\in\tau_G\}$, since the operators of union and complementary commute with $\pi$ because $\pi$ is surjective.
    \item \label{ls:pi_compacts}
    The map $\pi$ is closed \cite[Lemma 8.2.2]{bi:dikr}.
    Therefore, if $G$ is locally compact, every compact set of $G_H$ is of the form $\pi(K)$ for some closed compact set $K\subseteq G$: it follows by \cite[Lemma 8.2.5(b)]{bi:dikr}, where the provided compact set $K$ is also closed.
    \item \label{ls:borel_equality}
    If $E\in \mathcal{B}(G)$, then $\pi^{-1}(\pi(E))=E$ by \ref{ls:G_topology}.
    Indeed, given $F\in \mathcal{B}(G_H)$ such that $E=\pi^{-1}(F)$, it holds $\pi^{-1}(\pi(E))=\pi^{-1}(\pi(\pi^{-1}(F)))= \pi^{-1}(F) =E$. In particular, given $E_1,E_2\in \mathcal{B}(G)$:
    \begin{enumerate}[label*=\arabic*.]
        \item[(v')] \label{ls:disjoint} if $E_1 \cap E_2 = \emptyset$, then  $\pi(E_1)\cap\pi(E_2)= \emptyset$;
        \item[(v'')] \label{ls:contained} if $\pi(E_1)\subseteq \pi(E_2)$, then $E_1\subseteq E_2$.
    \end{enumerate}
\end{enumerate}

\begin{remark}
    If $G$ is a locally compact group, $G_H$ is a locally compact Hausdorff group by statements \ref{ls:G_H_T2}, and then it has a Haar measure which is unique up to a multiplicative constant, by Theorems \ref{th:existence_H} and \ref{th:uniqueness_H}. 
\end{remark}

Now, using the definition of Haar measure that follows by Definition \ref{def:Radon}, one can check that the pullback measure built on a locally compact group through $\pi$ is a Haar measure: this is a constructive proof of its existence.

\begin{proposition}             \label{th:quotient}
    Let $G$ be a locally compact topological group; consider the topological group $G_H \coloneqq G / \overline{\{e_G\}}$ and let $\pi \colon G \to G_H$ be the quotient map. Let $\mu$ be a left (resp. right) Haar measure on $G_H$. For every $E\in \mathcal{B}(G)$ define $\overline{\mu}(E)\coloneqq \mu(\pi(E))$. Then $\overline{\mu}$ is a left (resp. right) Haar measure on $G_H$.
\end{proposition}
\begin{proof}
    The function $\overline{\mu}$ is well defined by statement \ref{ls:pi_open}, and it is clearly positive and nonzero. By statement (v'), $\sigma$-addictivity follows. Translation-invariance is given by the fact that $\pi$ is a group homomorphism and $\mu$ is translation-invariant. Local finiteness works because $\pi$ is continuous.
    
    Given $E\in \mathcal{B}(G)$, outer regularity is satisfied:
    \[
    \begin{split}
        \overline{\mu}(E) & = \mu(\pi(E)) = \inf\{\mu(V) \mid V \supseteq \pi(E), \, V\in\tau_{G_H}\} \\
        & = \inf\{\mu(\pi(U)) \mid \pi(U) \supseteq \pi(E), \, U\in\tau_G\} \\
        & = \inf\{\mu(\pi(U)) \mid U \supseteq E, \, U\in\tau_G\} \\
        & = \inf\{ \overline{\mu}(U) \mid U \supseteq E, \, U\in\tau_G\}
    \end{split}
    \]
    where the third equality follows by statement \ref{ls:pi_open}, and the forth one by statement (v''). 

    Given $U\in\tau_G$, inner regularity is satisfied:
    \[
    \begin{split}
        \overline{\mu}(U) & = \mu(\pi(U)) = \sup\{\mu(C) \mid C \subseteq \pi(U) \text{, $C$ compact in } G_H\} \\
        & = \sup\{\mu(\pi(K)) \mid \pi(K) \subseteq \pi(U) \text{, $K$ closed and compact in } G \} \\
        & = \sup\{ \mu(\pi(K)) \mid K \subseteq U \text{, $K$ closed and compact in } G \} \\
        & = \sup\{ \overline{\mu}(K) \mid K \subseteq U \text{, $K$ closed and compact in } G \}
    \end{split}
    \]
    where the third equality follows by statement \ref{ls:pi_compacts}, and the forth one follows by statement (v'').
\end{proof}

\medskip

\begin{example}             \label{ex:product.1}
    Consider the topological group $X\times G$ as in Example \ref{ex:product}. Since $X\times G / \overline{\{(e_x,e_G)\}} \simeq X$ as topological groups, the Haar measure described in Example \ref{ex:product} coincides with the measure defined via quotient map in Theorem \ref{th:quotient}.
\end{example}

We conclude the paper giving an alternative proof of Theorem \ref{th:uniqueness}, which is based on the construction via quotients.

Given a locally compact topological group $G$, consider the topological group $G_H = G / \overline{\{e_G\}}$, endowed with the quotient topology as before, and $\pi \colon G \to G_H$ the quotient map. Denote with $\mathfrak{M}_H(G)$ and $\mathfrak{M}_H(G_H)$ the sets of left Haar measures on $G$ and $G_H$ respectively.
We consider the following functions: 
\[
\pi_* \colon \mathfrak{M}_H(G) \to \mathfrak{M}_H(G_H), \quad \mu \mapsto \pi_*\mu
\]
\[
\pi^* \colon \mathfrak{M}_H(G_H) \to \mathfrak{M}_H(G), \quad \nu \mapsto \pi^*\nu
\]
defined by:
\[
\pi_*\mu(F) = \mu(\pi^{-1}(F)) \quad \text{for every } F\in\mathcal{B}(G_H)
\]
\[
\pi^*\nu(E) = \nu(\pi(E)) \quad \text{for every } E\in\mathcal{B}(G)
\]

\begin{remark}
    For every $\nu \in \mathfrak{M}_H(G_H)$, the image $\pi^*\nu$ is the \emph{pullback} measure $\overline{\nu}$ introduced in Theorem \ref{th:quotient}.

    For every $\mu \in \mathfrak{M}_H(G)$, the image $\pi_*\mu$ is the well known \emph{pushforward} measure induced by the continuous map $\pi$: by an argument similar to the one given in Theorem \ref{th:quotient}, it is immediate to see that $\pi_*\mu$, which is a measure in general, is actually a left Haar measure.
\end{remark}

\begin{proposition}             \label{th:quotient_uniq}
    Let $G$ be a locally compact topological group; consider the topological group $G_H \coloneqq G / \overline{\{e_G\}}$ and let $\pi \colon G \to G_H$ be the quotient map. Then the functions $\pi_*$ and $\pi^*$ define a one-to-one correspondence between $\mathfrak{M}_H(G)$ and $\mathfrak{M}_H(G_H)$. 
\end{proposition}
\begin{proof}
    For every measure $\mu \in \mathfrak{M}_H(G)$ and for every set $E\in\mathcal{B}(G)$, one directly computes that $\pi^*\pi_*\mu(E) = \pi_*\mu(\pi(E)) = \mu(\pi^{-1}(\pi(E))) = \mu(E)$, since $\pi^{-1}(\pi(E)) = E$ by statement \ref{ls:borel_equality}. For every $\nu \in \mathfrak{M}_H(G_H)$ and for every $F\in\mathcal{B}(G_H)$ one has $\pi_*\pi^*\nu(F) = \pi^*\nu(\pi^{-1}(F)) = \nu(\pi(\pi^{-1}(F))) = \nu(F)$. Therefore, the functions $\pi_*$ and $\pi^*$ are inverses of each other.
\end{proof}

\medskip

Theorem \ref{th:uniqueness} can be proved as a corollary of Proposition \ref{th:quotient_uniq}.

\begin{proof}[Alternative proof of Theorem \ref{th:uniqueness}]
    Given $\mu, \mu' \in \mathfrak{M}_H(G)$, by Theorem \ref{th:uniqueness_H} there exists a constant $a>0$ such that $\pi_*\mu'=a\pi_*\mu$. Then, one has that $\mu' (E)= \pi^*\pi_*\mu' (E) = \pi^*a\pi_*\mu(E) = a\pi_*\mu(\pi(E)) = a \mu (E)$ for every $E\in \mathcal{B}(G)$, which implies $\mu' = a \mu$.
\end{proof}


\vspace{0.5cm}

\noindent \rule{\textwidth}{0.4pt}

\paragraph{Acknowledgements}
I am grateful to Professor Arvid Perego (Università di Genova) for his encouragement and advice, and I sincerely thank the referee for their precious comments.

\bibliographystyle{plainurl}
\bibliography{bibliography.bib}

\vspace{0.5cm}

\noindent \rule{\textwidth}{0.4pt}

\vspace{0.5cm}

\end{document}